\documentclass{amsproc}%
\usepackage{amsfonts}
\usepackage{amsmath}
\usepackage{amssymb}
\usepackage{graphicx}
\usepackage{comment}%
\providecommand{\U}[1]{\protect\rule{.1in}{.1in}}
\providecommand{\U}[1]{\protect\rule{.1in}{.1in}}
\includecomment{versiona}
\newtheorem{theorem}{Theorem}
\theoremstyle{plain}

\newtheorem{lemma}{Lemma}

\theoremstyle{definition}

\newtheorem{example}{Example}

\numberwithin{equation}{section}
\begin{document}
\title{Tilings from Tops of Overlapping Iterated Function Systems}
\author[M. F. Barnsley]{Michael F. Barnsley}
\author[C. de Wit]{Corey de Wit}

\begin{abstract}
The \textit{top} of the attractor of a hyperbolic iterated function system
$\left\{  f_{i}:\mathbb{R}^{n}\rightarrow\mathbb{R}^{n}|i=1,2,\dots,M\right\}
$ is defined and used to extend self-similar tilings to overlapping systems.
The theory provides sequences of approximate supertiles that converge to
tilings. Individual tiles in a tiling are limits of nested decreasing
sequences of approximate tiles. Examples include systems of finite type,
tilings related to aperiodic monotiles, and ones where there are infinitely
many distinct but related prototiles.

\end{abstract}
\dedicatory{ }\maketitle

\section{Introduction}

How can self-similar tiling theory \cite{bandt0, kenyon} be extended to
collections of overlapping tiles?

Tilings associated with iterated function systems (IFS) that obey the open set
condition (OSC) or are of finite type have been investigated in other papers,
see for example \cite{bandt2, bandt3, tilings, gdifs} and references therein.

Here we use top addresses, defined in Section \ref{tops section}, to extend
self-similar tiling theory to include tilings generated by iterated function
systems (IFS) that do not obey the open set condition and are not of finite
type. The tiles in these tilings are closures of the limits of decreasing
sequences of approximate tiles and have nonempty interiors. Individual tiles
become fully defined after finitely many steps, while the tilings themselves
become fully defined only in the infinite limit. Since our construction
depends mainly on topology rather than geometry, it also applies to
point-fibred IFSs acting on compact Hausdorff spaces, see \cite{tilings,
kieninger}.

In \cite{BB} a special case of tilings defined by top adddresses was
introduced. This paper extensively generalizes and extends this earlier result.

A different approach to the construction of tilings using overlapping tiles
has been provided by Akiyama and co-workers \cite{akiyama2}.

As an application, we show how our construction applies in the context of
aperiodic monotiles \cite{smith}.

\section{Iterated Function Systems, Tops, Blowups, and Tilings}

\subsection{Iterated function systems}

Let $F$ be an iterated function system (IFS)
\[
F=\{f_{i}:\mathbb{R}^{n}\mathbb{\rightarrow R}^{n}|i=1,2,\dots,M\}
\]
where $f_{i}:\mathbb{R}^{n}\mathbb{\rightarrow R}^{n}$ is a strictly
contractive homeomorphism with metric $d$ such that
\[
d(f_{i}(x),f_{i}(y))\leq\lambda d(x,y),\text{ some }0<\lambda<1,\text{ for all
}x,y,i.
\]
Then it is known \cite{hutchinson} that $F$ possesses a unique attractor, the
only non-empty compact subset of $\mathbb{R}^{n}$ which obeys
\[
A=\bigcup\limits_{i=1}^{M}f_{i}(A).
\]
We will use the fact that if $A$ has non-empty interior, then it is the
closure of its interior, $A=\overline{A^{\circ}}.$

We will need the following notions related to symbolic handling of subsets of
$A.$ Let $\mathbb{N}$ be the positive integers. Consider $\left\{
1,2,\dots,M\right\}  ^{\mathbb{N}}$, the set of infinite strings of the form
$\mathbf{j}=j_{1}j_{2}\dots$ where each $j_{i}$ belongs to $\left\{
1,2,\dots,M\right\}  $. We define $i:\left\{  1,2,\dots,M\right\}
^{\mathbb{N}}\rightarrow\left\{  1,2,\dots,M\right\}  ^{\mathbb{N}}$ for
$i\in\{1,2,\dots,M\}$ by $i(\mathbf{j)=}ij_{1}j_{2}\dots$ We may also write
$k_{1}k_{2}\dots k_{l}\mathbf{j}$ to mean the string $k_{1}k_{2}\dots
k_{l}j_{1}j_{2}\dots.$ The string $\mathbf{j}$ truncated to length $m$ is
denoted by $\mathbf{j}|m=j_{1}j_{2}\dots j_{m},$ and we define
\[
f_{\mathbf{j}|m}=f_{j_{1}}f_{j_{2}}\dots f_{j_{m}}=f_{j_{1}}\circ f_{j_{2}%
}\circ\dots\circ f_{j_{m}}.
\]
We define a metric $d^{\prime}$ by $d^{\prime}(\mathbf{j},\mathbf{k}%
)=2^{-\max\{n|j_{m}=k_{m},m=1,2,...,n\}}$ for $\mathbf{j}\neq\mathbf{k}$, so
that $\left(  \left\{  1,2,\dots,M\right\}  ^{\mathbb{N}},d^{\prime}\right)  $
is a compact metric space.

Then a continuous surjection $\pi:\left\{  1,2,\dots,M\right\}  ^{\mathbb{N}%
}\rightarrow A$ is defined by%
\[
\pi(\mathbf{j})=\lim_{m\rightarrow\infty}f_{\mathbf{j}|m}(x)=\lim
_{m\rightarrow\infty}f_{j_{1}}f_{j_{2}}\dots f_{j_{m}}(x).
\]
It is well-known \cite{hutchinson} that the limit is independent of $x$. Also,
the convergence is uniform in $\mathbf{j}$ over $\left\{  1,2,\dots,M\right\}
^{\mathbb{N}},$ and uniform in $x$ over any compact subset of $\mathbb{R}^{n}%
$. We say $\mathbf{j}\in\Sigma$ is an \textit{address} of the point
$\pi(\mathbf{j})\in A$. Each point in $A$ has at least one address.

\subsection{\label{tops section}Tops}

Since $\pi:\left\{  1,2,\dots,M\right\}  ^{\mathbb{N}}\rightarrow A$ is
continuous and onto, it follows that $\pi^{-1}(x)$ is closed for all $x\in A.$
Hence, a map $\tau:A\rightarrow\left\{  1,2,\dots,M\right\}  ^{\mathbb{N}}$
and a set $\Sigma$ are well-defined by
\begin{align*}
\tau(x)  &  :=\max\{\mathbf{k\in}\Sigma|\pi(\mathbf{k)=}x\},\\
\Sigma &  :=\tau(A)=\left\{  \tau(x):x\in A\right\}  ,
\end{align*}
where the maximum is with respect to lexicographical ordering. We call
$\tau(x)$ the \textit{top }address\textit{ }of $x\in A$.

The top address $\tau(x)$ of $x\in A$ and the set of top addresses $\Sigma$
may be calculated by following orbits under the dynamical system\textit{
}$D:A\rightarrow A$ as follows. Partition $A$ into $\left\{  A_{i}\right\}  $
according to
\[
A_{1}=f_{1}(A),A_{2}=f_{2}(A)\backslash A_{1},\ A_{3}=f_{3}(A)\backslash
(A_{1}\cup A_{2}),\dots,A_{M}=f_{M}(A)\backslash\cup_{n\neq M}A_{n}.
\]
We always assume $A_{M}\neq\emptyset.$ Define the orbit $\left\{
x_{n}\right\}  _{n=1}^{\infty}$ of $x=x_{1}\in A,$ by $x_{n+1}=Dx_{n}%
=f_{i_{n}}^{-1}(x_{n})$ where $i_{n}$ is the unique index such that $x_{n}\in
A_{i_{n}}$. Then $\tau(x)=i_{1}i_{2}i_{3}\dots\in\Sigma$.

Top addresses are related to $\beta$-expansions, see in particular
\cite{thurston}. See also \cite{akiyama} and references therein.

Define the \textit{critical set }of $F$ to be
\[
C=\overline{%
{\textstyle\bigcup\limits_{i}}
\partial A_{i}\backslash\partial A}.
\]

Let $\sigma:\{1,2,\dots,M\}^{\mathbb{N}}\rightarrow\{1,2,\dots,M\}^{\mathbb{N}%
}$ be the shift operator defined by $\sigma(\mathbf{j)=}j_{2}j_{3}\dots$. Then%
\[
f_{i}\circ\pi=\pi\circ i\mathbf{\ }\text{and }f_{j_{1}}^{-1}\circ\pi\left(
\mathbf{j}\right)  =\pi\circ\sigma\left(  \mathbf{j}\right)  \text{ for all
}i\text{ and }\mathbf{j}.
\]
Key to the results in this paper is the total shift invariance of $\Sigma$, as follows.

\begin{lemma}
\label{proposition1}$\Sigma=\sigma\left(  \Sigma\right)  $.
\end{lemma}

\begin{proof}
This is readily checked \cite{BB}.
\end{proof}

\subsection{Tops at finite depth}

Define $\Sigma_{n}$ to be the elements of $\Sigma$ truncated to length
$n\in\mathbb{N}$,
\[
\Sigma_{n}=\left\{  \left(  \mathbf{j}|n\right)  |\mathbf{j}\in\Sigma\right\}
.
\]
Corresponding to $\mathbf{i}|n\in$ $\Sigma_{n}$, define members of a partition
of $A$ at depth $n$ by
\[
A_{\mathbf{i}|n}=\left\{  \pi(\mathbf{j}):\mathbf{j\in}\Sigma,\left(
\mathbf{j|}n\right)  =\left(  \mathbf{i|}n\right)  \right\}  .
\]
Equivalently,
\begin{equation}
A_{\mathbf{i}|n}=f_{\mathbf{i}|n}(A)\backslash\cup\{f_{\mathbf{j}%
|n}(A)|\mathbf{j\in}\Sigma,(\mathbf{j}|n)>(\mathbf{i}|n)\}.\label{coreydef}%
\end{equation}

We will need the following lemmas.

\begin{lemma}
\label{topshiftlemma}Let $\mathbf{k}\in\Sigma.$ Then $f_{k_{1}}\left(
A_{(\sigma\left(  \mathbf{k}\right)  \mathbf{)|}n-1}\right)  \supseteq
A_{\mathbf{k|}n}$ for all $n\in\mathbb{N}$, $n>1$.
\end{lemma}

\begin{proof}
We compare the sets
\[
\{f_{k_{1}...k_{n}}(x)|f_{k_{1}...k_{n}}(x)\notin f_{l_{1}l_{2}...l_{n}%
}(A)\text{ for all }l_{1}...l_{n}>k_{1}...k_{n}\}
\]
and
\[
\{f_{k_{1}...k_{n}}(x)|f_{k_{1}}f_{k_{2}...k_{n}}(x)\notin f_{k_{1}}%
f_{l_{2}...l_{n}}(A)\text{ for all }l_{2}...l_{n}>k_{2}...k_{n}\}.
\]
The condition in the latter expression is less restrictive.
\end{proof}

Also, since $\Sigma$ is shift invariant, we have:

\begin{lemma}
\label{smallerlemma}Let $n>1.$ If $i_{1}i_{2}\dots i_{n-1}i_{n}\in\Sigma_{n}$,
then both $i_{2}\dots i_{n-1}i_{n}\ $and $i_{1}i_{2}\dots i_{n-1}\ $belong to
$\Sigma_{n-1}.$
\end{lemma}

\subsection{Top blowups and tilings\label{topblowupsec}}

We focus on the overlapping case, where the OSC$\ $does not hold, with
particular interest in the case where $A$ has nonempty interior. We show how
tilings of $\mathbb{R}^{n}$ can be defined, generalizing well known
constructions that apply when the OSC holds \cite{bandt2, tilings, strichartz}.

Let $\overleftarrow{\Sigma}\subset\{1,2,\dots,M\}^{\mathbb{N}}$ be the set
strings $\mathbf{l}=l_{1}l_{2}\dots$ such that $l_{n}l_{n-1}\dots l_{1}%
\in\Sigma_{n}$ for all $n.$ Note that $\overleftarrow{\Sigma}$ is closed but
$\Sigma$ may not be. We define, for $\mathbf{j\in}\overleftarrow{\Sigma}$ and
$k\in\mathbb{N}$,
\[
f_{-\mathbf{j}|k}=f_{j_{1}}^{-1}f_{j_{2}}^{-1}\dots f_{j_{k}}^{-1}=f_{j_{1}%
}^{-1}\circ f_{j_{2}}^{-1}\circ\dots\circ f_{j_{k}}^{-1}.
\]

Strichartz \cite{strichartz} defines the \textit{blowup} $\mathcal{A}\left(
\mathbf{j}\right)  $ of $A$ corresponding to $\mathbf{j\in}%
\overleftarrow{\Sigma}$ by%
\[
\mathcal{A}\left(  \mathbf{j|}n\right)  =\bigcup\limits_{l=1}^{n}%
f_{-\mathbf{j}|l}\left(  A\right)  \text{ and }\mathcal{A}\left(
\mathbf{j}\right)  =\bigcup\limits_{n=1}^{\infty}\mathcal{A}\left(
\mathbf{j|}n\right)  .
\]
(Note that Strichartz was concerned with situations where the OSC holds, in
which case $\overleftarrow{\Sigma}=\{1,2,\dots\}^{\mathbb{N}}.$) Here the
unions are of increasing nested sequences of sets so $\mathcal{A}\left(
\mathbf{j|}n\right)  =f_{\mathbf{j}|n}^{-1}\left(  A\right)  .$ Note that
$\mathcal{A}\left(  \mathbf{j|}n\right)  $ is related to $\mathcal{A}\left(
\mathbf{i|}n\right)  $ by the isometry $\left(  f_{-\mathbf{j}|n}\right)
\left(  f_{-\mathbf{i}|n}\right)  ^{-1}$. But possible relationships between
$\mathcal{A}\left(  \mathbf{i}\right)  $ and $\mathcal{A}\left(
\mathbf{j}\right)  $ are subtle because inverse limits are involved.

Under conditions on $F$ and fixed $\mathbf{j}\in\overleftarrow{\Sigma}$,
stated in Theorem \ref{tiling theorem}, we construct a collection of sets as
follows. We define:%
\begin{align}
\Pi(\mathbf{j}|k)  &  :=\left\{  f_{-\mathbf{j}|k}(A_{\mathbf{t|}%
k+1})|\mathbf{t\in}\Sigma\right\}  ,\label{tile limit}\\
\Pi(\mathbf{j})  &  =\lim_{k\rightarrow\infty}\Pi(\mathbf{j}|k),\text{ as
defined in Theorem \ref{tiling theorem}.}\nonumber
\end{align}
For $\mathbf{t\in}\Sigma$ and $\mathbf{j}\in\overleftarrow{\Sigma},$ we may
refer to sets of the form $f_{-\mathbf{j}|k}(A_{\mathbf{t|}k+1})$ and
$f_{-\mathbf{j}|k}(\overline{A_{\mathbf{t|}k+1}}),$ where the bar denotes
closure, as \textit{pre-tiles }and \textit{tiles} respectively.

\begin{versiona}
\begin{figure}[ptb]%
\centering
\includegraphics[
height=2.0 in,
width=5.0 in
]%
{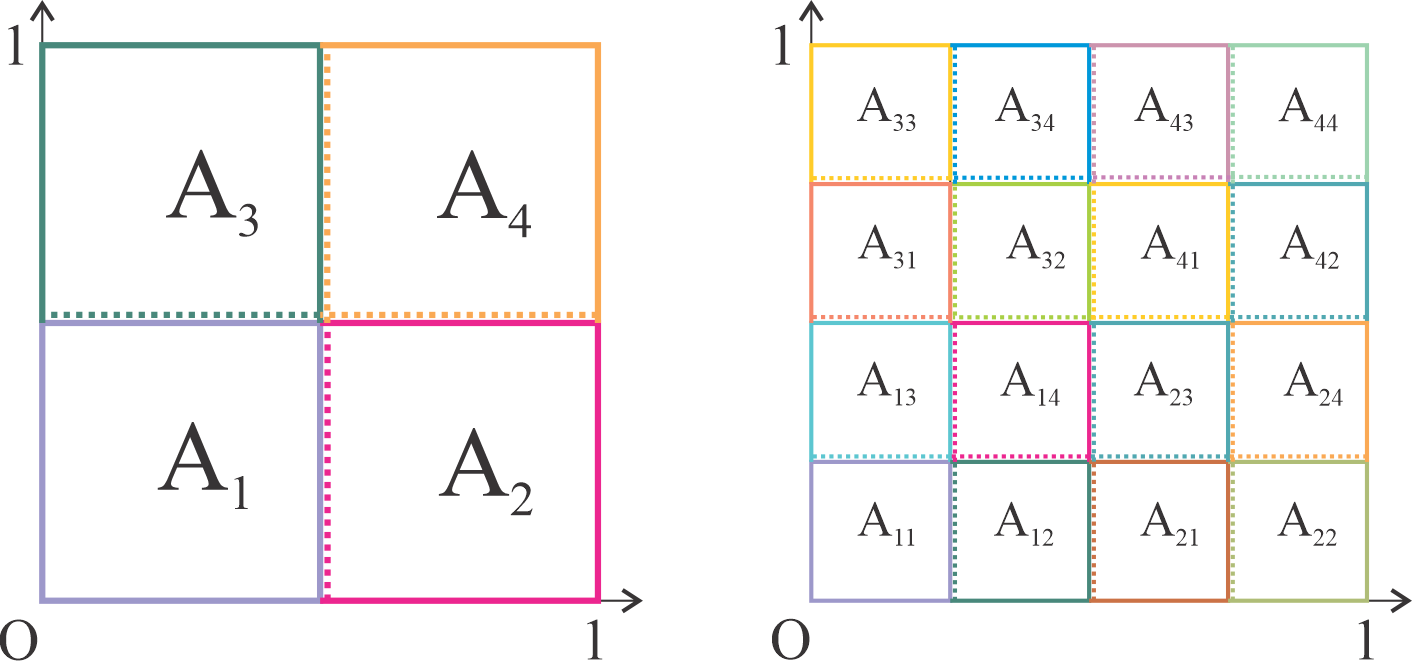}%
\caption{The right hand panel represents the
attractor of the IFS in Equation \ref{eqifs1} partitioned at depth one into
four pre-tiles, indicated by different colours. Dotted lines indicate open
boundaries. The left hand panel is similar, but partitioned to depth two.}%
\label{ownbound}%
\end{figure}
\end{versiona}

\begin{versiona}
\begin{figure}[ptb]%
\centering
\includegraphics[
height=3.0 in,
width=5.0 in
]%
{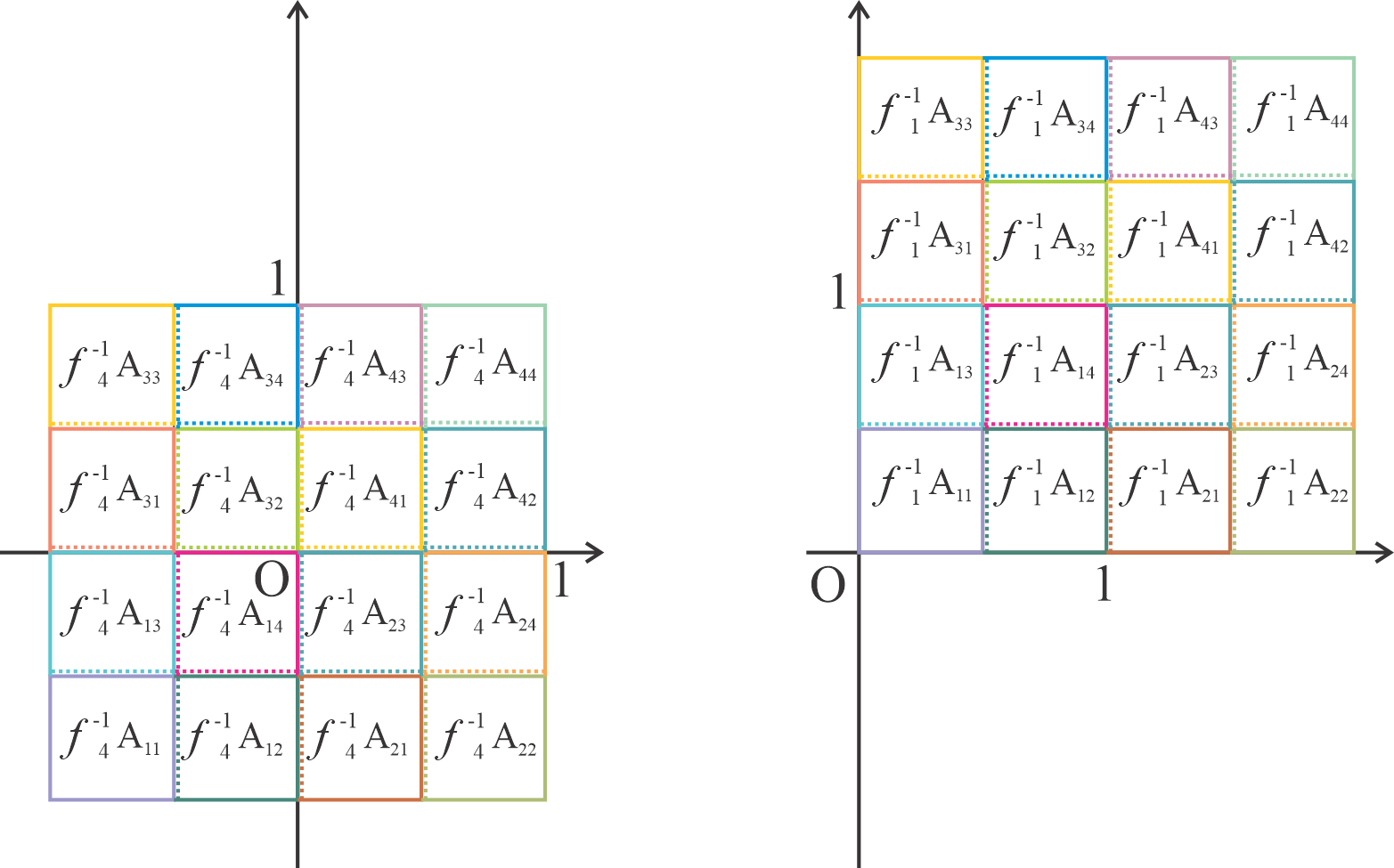}%
\caption{The maps $f_{4}^{-1}$ (left) and $f_{1}^{-1}$ (right) have been
applied to the right-hand partition in Figure \ref{ownbound}. Notice the
different positionings of the closed and open edges of the pre-tiles that
cover $\left(  0,1\right) \times\left(0,1\right).$}%
\label{ownbgrid}%
\end{figure}
\end{versiona}

When the OSC holds and $A$ has nonempty interior, that is $A^{\circ}%
\neq\emptyset$, the limit in Equation \ref{tile limit} is well defined because
$\overline{A_{\mathbf{t|}k+1}}=f_{\mathbf{t|}k+1}(A)$ so that, up to closure
of the sets $A_{\mathbf{t|}k+1}$, we have $\Pi(\mathbf{j}|k)\subset
\Pi(\mathbf{j}|k+1)$ for all $k,$ so $(\left\{  f_{-\mathbf{j}|k}%
(\overline{A_{\mathbf{t|}k+1}})|\mathbf{t\in}\Sigma\right\}  )_{k=1}^{\infty}$
is an increasing nested sequence of finite collections of just-touching sets.
In this case, at each level new pre-tiles are added, while the tiles defined
by taking the closures of pre-tiles belonging to the previous levels remain
unchanged. In Section \ref{osc case} we provide a simple example of pre-tiles
and use of the notion of the top of an attractor. Tilings constructed in this
manner, though with different language, are discussed in \cite{polygon,
graphtilings}.

But the general situation is quite different because nested sequences of
pre-tiles may change substantially from one level to the next. As we will
show, the collection of sets $\Pi(\mathbf{j})$ is well defined in situations
where $\Pi(\mathbf{j}|k)$ may not be a subset of $\Pi(\mathbf{j}|k+1)$, even
when the closures of the pre-tiles is taken. In general, for $\mathbf{j}%
\in\overleftarrow{\Sigma}$, here is an informal preview of what happens. At
each successive value of $k$, the pre-tiles from the previous level are
smaller, and new pre-tiles are added. Sequences of sucessively smaller
pre-tiles cease to change after finitely many steps, and their closures remain
thereafter as established tiles with nonempty interiors, members of a tiling
that is fully defined only in the limit as $k\rightarrow\infty$. Two simple
examples of this process, illustrating our notion of pre-tiles and showing how
pre-tiles may be successively smaller but finally stabilize, are given in
Section \ref{sec3}. The reader may prefer to consult Section \ref{sec3} before
proceeding to our main result, Theorem 1.

In the overlapping case it is not obvious that the pre-tiles we will define
have non-empty interiors. We will need the following Lemma.

\begin{lemma}
\label{lemmafour} Let $E$ and $F$ be subsets of $\mathbb{R}^{n}$ such that
$E\backslash F$ is not empty, $E$ is the closure of its interior, and $F$ is
closed. Then the interior $\left(  E\backslash F\right)  ^{\circ}$ of
$E\backslash F$ is not empty.

\begin{proof}
Suppose $\left(  E\backslash F\right)  ^{\circ}$ is empty. Then $E^{\circ
}\subset F.$ Taking the closure of both sides, it follows that $E$ is
contained in $F$. That is, $E\backslash F$ is empty, a contradiction. So if
$E\backslash F$ is not empty, then $\left(  E\backslash F\right)  ^{\circ}$ is
not empty.
\end{proof}
\end{lemma}

In Theorem \ref{tiling theorem} $\mathbf{j\in}\overleftarrow{\Sigma}$ is fixed
throughout. We consider expressions of the form%
\[
S(l,\mathbf{i},k):=f_{j_{1}}^{-1}f_{j_{2}}^{-1}\dots f_{j_{k}}^{-1}%
(A_{j_{k}j_{k-1}\dots j_{l+2}j_{l+1}i_{1}i_{2}\dots i_{l+1}})
\]
for fixed $l$ and successively larger values of $k,$ where $k\geq l+1\geq1$
and $\mathbf{i\in}\Omega_{l}$ where
\[
\Omega_{l}:=\left\{  \mathbf{i}\in\Sigma:j_{k}j_{k-1}\dots j_{l+2}j_{l+1}%
i_{1}i_{2}\dots i_{l+1}\in\Sigma_{k+1}\text{ for all }k\geq l+1\right\}  .
\]
$\Omega_{l}$ is the set of addresses whose corresponding pre-tiles first
appear at the $l$-th iteration and remain visible for the rest of the tiling
construction. 

The sets $S(l,\mathbf{i},k)$ are supersets of pre-tiles that occur in the
expression $\lim_{k\rightarrow\infty}\Pi(\mathbf{j}|k)$. The key observation
underlying Theorem \ref{tiling theorem} is that, under an extra condition on
$\mathbf{j}$, as $k$ increases the sequence of sets $S(l,\mathbf{i},k)$
decreases for finitely many steps after which it does not change.

\begin{theorem}
\label{tiling theorem} Let $F$ be a strictly contractive IFS on $\mathbb{R}%
^{n}$ whose attractor $A$ has nonempty interior. Let $C$ be the critical set
of $F.$ Let $\mathbf{j\in}\overleftarrow{\Sigma}$, $\widehat{\lambda}%
\in(\lambda,1),$ $x_{0}\in A,$ and $\varepsilon>0,$ be fixed such that the
sequence $\left\{  x_{n}=f_{j_{n}}f_{j_{n-1}}\dots f_{j_{1}}(x_{0})\right\}
_{n=1}^{\infty}$ obeys $d(x_{n},C)>$ $\varepsilon\left(  \widehat{\lambda
}\right)  ^{n}$ for all $n$.

(i) For fixed $l\in\mathbb{N}$ and $\mathbf{i\in}\Omega_{l}$, the sequence
$\left\{  S(l,\mathbf{i},k):k=0,1,\dots\right\}  $ is nested decreasing
according to $S(l,\mathbf{i},0)\supseteq S(l,\mathbf{i},1)\supseteq\dots$ and
there is finite $K$ such that the sequence converges in a finite number of
steps to a set $S(l,\mathbf{i})$ with non-empty interior%
\[
S(l,\mathbf{i}):=S(l,\mathbf{i},K)=S(l,\mathbf{i},K+1)=\dots.
\]

(ii) Let%
\[
S(l):=\{S(l,\mathbf{i}):\mathbf{i}\in\Omega_{l}\}.
\]
Then $\left\{  S(l)\right\}  _{l=0}^{\infty}$ is a nested increasing sequence
of collections of sets with non-empty interiors, according to
\[
S(0)\subset S(1)\subset\dots.
\]

(iii) Define $\Pi(\mathbf{j})$ as the union limit of the above sequence, that
is
\[
\Pi(\mathbf{j}):=%
{\textstyle\bigcup\limits_{n=1}^{\infty}}
S(n).
\]
Then $\Pi(\mathbf{j})$ provides a tiling of a region that includes
\[
f_{j_{1}}^{-1}\circ f_{j_{2}}^{-1}\circ\dots\circ f_{j_{l}}^{-1}\circ
f_{j_{l+1}}^{-1}(A_{j_{l+1}})
\]
for all sufficiently large $l$. In particular, if $\mathbf{j}$ includes the
symbol $1$ infinitely often, then $\Pi(\mathbf{j})$ provides a tiling of a
region that includes the blowup $\mathcal{A}\left(  \mathbf{j}\right)  $.
\end{theorem}

\begin{proof}
We begin by noting the following. The sets $A_{j_{k}\dots j_{l+2}j_{l+1}%
i_{1}\dots i_{l+1}}$, as in the preamble to the theorem, are non-empty by
Lemmas 2 and 3. Each of these sets has non-empty interior, by the following
argument. $A$ has non-empty interior by assumption. Because $A$ is an
attractor of an IFS\ of invertible maps, it is the closure of its interior,
i.e. $A=\overline{(A^{\circ})}$. Since the composition of homeomorphisms
preserves closure and non-empty interior, by Equation \ref{coreydef} in
Section 2.3, $A_{j_{k}\dots j_{l+2}j_{l+1}i_{1}\dots i_{l+1}}$ is of the form
$E\backslash G$ where $E$ has non-empty interior and $G$ is closed, so
inductively $A_{j_{k}\dots j_{l+2}j_{l+1}i_{1}\dots i_{l+1}}$ has non-empty
interior by Lemma \ref{lemmafour}.

To prove the first part of (i), observe that by Lemma \ref{topshiftlemma},
\[
f_{j_{k+1}}(A_{j_{k}\dots j_{l+2}j_{l+1}i_{1}\dots i_{l+1}})\supseteq
A_{j_{k+1}j_{k}\dots j_{l+2}j_{l+1}i_{1}\dots i_{l+1}}%
\]
and apply $f_{j_{1}}^{-1}\dots f_{j_{k+1}}^{-1}$ to both sides.

For the remainder of (i), notice that if $f_{j_{k+1}}(A_{j_{k}\dots
j_{l+2}j_{l+1}i_{1}\dots i_{l+1}})\cap C=\emptyset,$ then for any
$f_{\omega_{1}\omega_{2}\dots\omega_{k+1}}(A)$ intersecting $f_{j_{k+1}%
}(A_{j_{k}\dots j_{l+2}j_{l+1}i_{1}\dots i_{l+1}})$ such that $\omega
_{1}\omega_{2}\dots\omega_{k+1}>j_{k+1}j_{k}\dots j_{l+2}j_{l+1}i_{1}\dots
i_{l+1}$, we must have $\omega_{1}=j_{k+1}$. That is, $f_{j_{k+1}}%
(A_{j_{k}\dots j_{l+2}j_{l+1}i_{1}\dots i_{l+1}})=A_{j_{k+1}j_{k}\dots
j_{l+2}j_{l+1}i_{1}\dots i_{l+1}},$ and by applying $f_{j_{1}}^{-1}f_{j_{2}%
}^{-1}\dots f_{j_{k+1}}^{-1}$ to both sides we get $S(l,\mathbf{i,}%
k)=S(l,\mathbf{i,}k+1).$ So it suffices to find a large enough $K$ such that
$A_{j_{k}\dots j_{l+2}j_{l+1}i_{1}\dots i_{l+1}}$ contains no points of $C$
for all $k\geq K$.

But $A_{j_{k}\dots j_{l+2}j_{l+1}i_{1}\dots i_{l+1}}$ is contained in
$f_{j_{k+1}j_{k}\dots j_{l+2}j_{l+1}}(A)$ and for any fixed $l$ and $k$
sufficiently larger than $l,$ the condition on $\mathbf{j\in}%
\overleftarrow{\Sigma}$, $\widehat{\lambda}\in(\lambda,1),$ $x_{0}\in A,$ and
$\varepsilon>0,$ ensures that $f_{j_{k+1}j_{k}\dots j_{l+2}j_{l+1}}(A)$ does
not meet $C.$ Indeed,%
\begin{equation}
d(f_{j_{k+1}j_{k}\dots j_{l+2}j_{l+1}l+1\dots j_{1}}(x_{0}),f_{j_{k+1}%
j_{k}\dots j_{l+2}j_{l+1}}(A))\leq\lambda^{k-l}|A|, \label{estimate1}%
\end{equation}
where $\left\vert A\right\vert =\max\left\{  d(x,y)|x,y\in A\right\}  $ with
$d$ the metric on $\mathbb{R}^{n},$ and%
\begin{equation}
d(f_{j_{k+1}j_{k}\dots j_{l+2}j_{l+1}\dots j_{1}}(x_{0}),C)\geq\left(
\widehat{\lambda}\right)  ^{k+1}\varepsilon. \label{estimate2}%
\end{equation}
So the shortest distance between $f_{j_{k+1}j_{k}\dots j_{l+2}j_{l+1}}(A)$ and
$C$ is greater than
\begin{equation}
\left(  \widehat{\lambda}\right)  ^{k+1}\varepsilon-\lambda^{k-l}|A|
\label{ineq}%
\end{equation}
which is strictly positive, for any fixed $l$ and any $k$ that is sufficienly
larger than $l.$

To prove (ii), namely that $S(l)\subset S(l+1),$ note that
\begin{align*}
S(l)  &  =\{S(l,\mathbf{i}):\mathbf{i}\in\Omega_{l}\}\\
&  =\{f_{j_{1}}^{-1}\dots f_{j_{k}}^{-1}(A_{j_{k}\dots j_{l+2}j_{l+1}%
i_{1}\dots i_{l+1}}):\text{for all }k\text{ as in (i)},\text{ }\mathbf{i}%
\in\Omega_{l}\},
\end{align*}
while
\begin{align*}
S(l+1)  &  =\{S(l+1,\mathbf{i}):\mathbf{i}\in\Omega_{l+1}\}\\
&  =\{f_{j_{1}}^{-1}\dots f_{j_{k}}^{-1}(A_{j_{k}\dots j_{l+3}j_{l+2}%
i_{1}\dots i_{l+2}}):\text{for all }k\text{ as in (i)},\mathbf{i}\in
\Omega_{l+1}\}.
\end{align*}
Fix $k$ sufficiently large, say $k\geq K,$ so that both expressions on the
right hold. Then, using Lemma \ref{smallerlemma}, the set of allowed indices
in the first expression, ones with tails of the form $j_{l+1}i_{1}\dots
i_{l+1},$ where $l+1$ is fixed, is a subset of the set of allowed indices in
the second expression, with tails of the form $i_{1}\dots i_{l+2}$. So
$S(l)\subset S(l+1)$. This completes the proof of (ii).

To prove (iii) we make the following observations. (a) Let $K_{l}$ be the
maximum value of
\[
K_{l}:=\min\{k>l:f_{j_{k+1}j_{k}\dots j_{l+2}j_{l+1}i_{1}\dots i_{l+1}}(A)\cap
C=\varnothing\text{ }\forall\text{ }\mathbf{i}\in\Omega_{l}\}.
\]
Then it follows from (i) that the $k>l+1$ in the definition of $\Omega_{l}$
can be replaced by $l+1\leq k\leq K_{l}.$ (b) It can be seen from Equation
\ref{ineq} that $\left\{  K_{l}\right\}  $ is a non-increasing. In fact, upon
expansion, for%
\begin{equation}
l>(\log_{\widehat{\lambda}/\lambda}(|A|/\varepsilon\lambda)-1)/(1+\log
_{\widehat{\lambda}/\lambda}(\lambda)) \label{estimate3}%
\end{equation}
we have $K_{l}=l+1$.

It follows that
\begin{align*}%
{\textstyle\bigcup}
S(l)  &  =%
{\textstyle\bigcup\limits_{i\in\Omega_{l}}^{\infty}}
f_{j_{1}}^{-1}\dots f_{j_{K_{l}}}^{-1}(A_{j_{K_{l}}\dots j_{l+1}i_{1}\dots
i_{l+1}})\\
&  =f_{-\mathbf{j}|l}(f_{-\sigma^{l}(\mathbf{j)|}K_{l}-l}(A_{j_{K_{l}}\dots
j_{l+1}}))\text{ by (a)}\\
&  =f_{-\mathbf{j}|l}(f_{j_{l+1}}^{-1}(A_{j_{l+1}}))\text{ for all large
enough }l\text{, as in Equation \ref{estimate3}.}%
\end{align*}
The penultimate statement in the theorem follows on noting that $f_{1}%
^{-1}(A_{1})=A$ and that $\left\{  f_{-\mathbf{j}|l}(A)\right\}
_{l=1}^{\infty}$ is a nested increasing sequence. In particular, the sequence
contains $A,$ and in particular the fixed-point of $f_{1},$ for all large
enough $l,$ and is acted on by $f_{1}^{-1}$ infinitely many times, so it
contains $f_{1}^{-1}(A),$ $f_{1}^{-1}\circ f_{1}^{-1}(A),\dots.$
\end{proof}

\begin{example}
\label{example1}The tiling $\Pi(\overline{1})$ is well-defined and its support
is the blowup ${\textstyle\bigcup\limits_{k}}f_{1}^{-k}(A).$ In this case the
entire tiling is provided by the collection of pretiles $%
{\textstyle\bigcup\limits_{k}}
f_{1}^{-k}(\{A_{\mathbf{i}|(k+1)}:$ $\mathbf{i\in}\Sigma\mathbf{\},}$ because
the sequence $\{f_{1}^{-k}(\{A_{\mathbf{i}|(k+1)}:$ $\mathbf{i\in}%
\Sigma\mathbf{\}\}}_{k=1}^{\infty}$ is nested increasing. The key point is
that $\Omega_{l}=\Sigma_{l+1}$ and because of this, the blowup of every set in
the top-partitioned attractor at depth $k$ is included in the final tiling.
Examples of this kind were introduced in \cite{BB}.
\end{example}

\begin{example}
If $\mathbf{j\in}\overleftarrow{\Sigma}$ is periodic and contains the symbol
$1,$ then the tiling $\Pi(\sigma^{m}\mathbf{j})$ is well-defined for all
$m=0,1,2\dots$.
\end{example}

\begin{example}
If $\mathbf{j\in}\overleftarrow{\Sigma}$ contains the symbol $k$ infinitely
often, and $A_{k}$ contains the fixed point of $f_{k}$ in its interior, then
$\Pi(\mathbf{j})$ tiles $\mathbb{R}^{n}.$
\end{example}

\section{\label{sec3}NOTION\ OF\ PRE-TILES AND\ TOPS}

\begin{versiona}
\begin{figure}[ptb]%
\centering
\includegraphics[
height=1.25 in,
width=5.0in
]%
{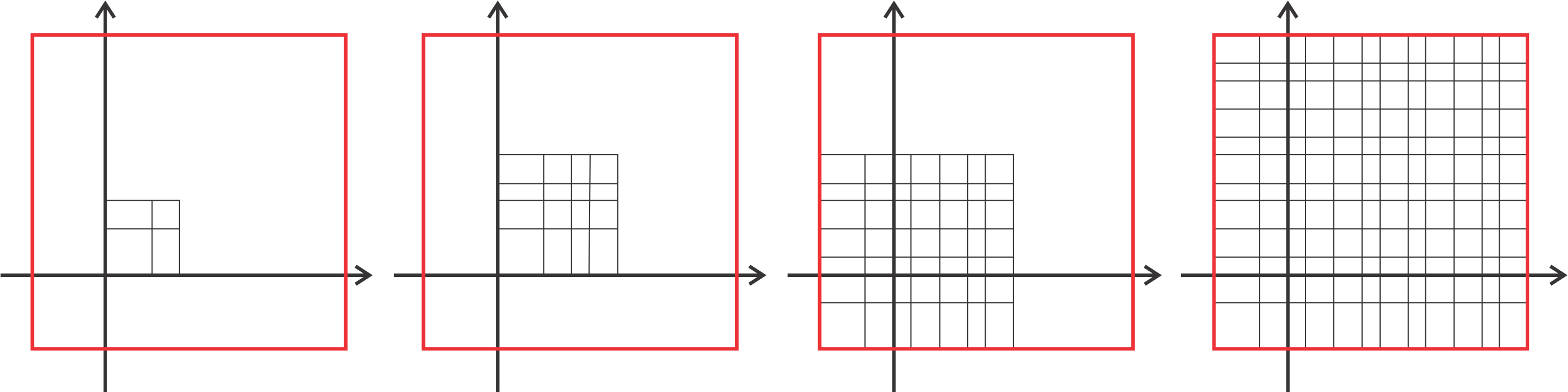}%
\caption{This picture relates to the IFS in Equation
\ref{ex2formula} and shows successive tilings of the form $f_{-\mathbf{k}|n}(\{\overline{A_{\mathbf{i}|(n+1)}}:\mathbf{i\in\Sigma\})}$ where
$\mathbf{k}|4=1414.$ In successive panels the bounding box shown in red
remains constant.}%
\label{plaintiles}%
\end{figure}
\end{versiona}

\begin{versiona}
\begin{figure}[ptb]%
\centering
\includegraphics[
height=2.0 in,
width=2.0 in
]%
{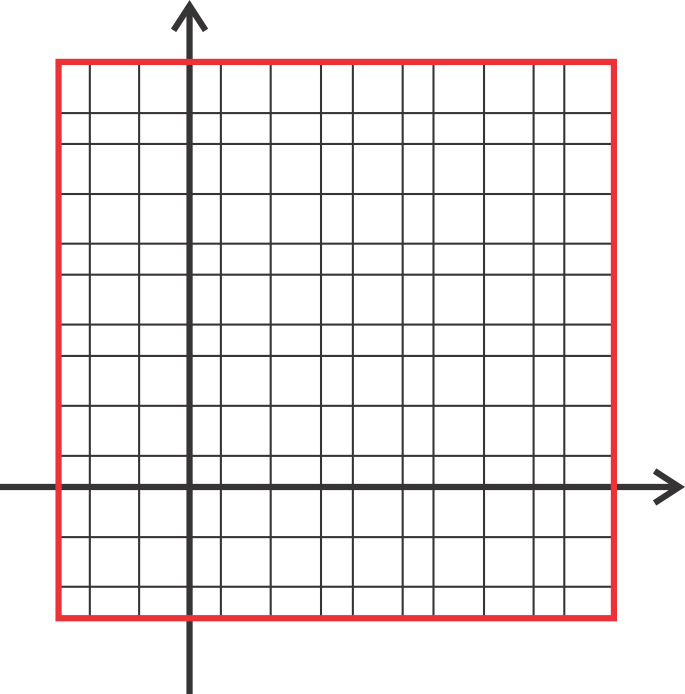}%
\caption{This shows the converged tiling, within the red
bounding box, following on from Figure \ref{plaintiles}. See text. As the tiling develops further, parts within the bounding box cease to change.}%
\label{plaintilesfinal}%
\end{figure}
\end{versiona}

\subsection{\label{osc case}Basic example of top tiling when the OSC holds}

Here we use a simple two-dimensional example, where the OSC holds, to
illustrate notions and language associated with top tilings. We consider the
IFS
\begin{equation}
\{\mathbb{C};f_{1}(z)=z/2,f_{2}(z)=(z+1)/2,f_{3}(z)=(z+i)/2,f_{4}%
(z)=(z+1+i)/2\} \label{eqifs1}%
\end{equation}
which has attractor $A=[0,1]\times\lbrack0,1].$ We may think of $A$ as being
tiled\ by the four half-sized copies $\left\{  f_{1}(A),f_{2}(A),f_{3}%
(A),f_{4}(A)\right\}  $ of $A.$ Each point of $A$ belongs to a unique
pre-tile, specifically
\[
A_{1}=[0,\frac{1}{2}]\times\lbrack0,\frac{1}{2}],A_{2}=(\frac{1}{2}%
,1]\times\lbrack0,\frac{1}{2}],A_{3}=[0,\frac{1}{2}]\times(\frac{1}%
{2},1],A_{4}=(\frac{1}{2},1]\times(\frac{1}{2},1].
\]
See Figure \ref{ownbound}. These pre-tiles, which generally are neither open
nor closed, but have non-empty interiors, are obtained by taking the
\textquotedblleft top\textquotedblright\ of the set $\{f_{1}(A),f_{2}%
(A),f_{3}(A),f_{4}(A)\}$, after assigning lexicographical ordering to the
indices $\{1,2,3,4\}$. Here the set $f_{1}(A)$ lies \textquotedblleft on top
of\textquotedblright\ parts of $f_{2}(A),f_{3}(A),$ and $f_{4}(A)$. What we
call the \textquotedblleft top tiling\textquotedblright\ of $A$ in this
example comprises four pre-tiles, the portions of the sets \ $f_{i}(A)$ that
belong to the \textquotedblleft top\textquotedblright.

We extend this top tiling to a larger region in two steps. (i)We use
\[
A=%
{\textstyle\bigcup_{i,j=1}^{4}}
f_{ij}(A)=%
{\textstyle\bigcup_{i,j=1}^{4}}
A_{ij},
\]
where $A_{ij}=f_{ij}(A)\backslash\cup_{st<ij}A_{st}.$ For example, as
illustrated in Figure \ref{ownbound},%
\[
A_{11}=[0,\frac{1}{4}]\times\lbrack0,\frac{1}{4}],A_{12}=(\frac{1}{4},\frac
{1}{2}]\times\lbrack0,\frac{1}{4}],A_{13}=(\frac{1}{2},\frac{3}{4}%
]\times\lbrack0,\frac{1}{4}],A_{14}(A)=(\frac{3}{4},1]\times\lbrack0,\frac
{1}{4}].
\]
That is, we have used the lexicographic ordering $11>12>13>14>21>\dots$ to
determine to which tile to assign points which belong to more than one set in
$\{f_{ij}(A)\}.$ Note that in general $A_{ij}\neq f_{ij}(A)$ because
$f_{ij}(A)$ is closed but $A_{ij}$ is typically not closed. (ii)We apply
$f_{1}^{-1}$ to $\left\{  A_{ij}\right\}  $ to obtain a partition of
$[0,2]\times\lbrack0,2]$ consisting of the sixteen disjoint pre-tiles
illustrated in the right-hand panel Figure \ref{ownbgrid},
\begin{align*}
f_{1}^{-1}A_{11}  &  =[0,0.5]\times(0,0.5],f_{1}^{-1}A_{12}=(0.5,1]\times
\lbrack0,0.5],\\
f_{1}^{-1}A_{13}  &  =[0.5,0.5]\times(0,1],f_{1}^{-1}A_{14}=(0.5,2]\times
\lbrack0,0.5],\\
&  \dots\\
f_{1}^{-1}A_{43}  &  =(1.5,2]\times(1,1.5],f_{1}^{-1}A_{44}=(1.5,2]\times
(1.5,2].
\end{align*}
This tiling extends the pre-tiling $\{A_{i}\},$ in particular
\[
f_{1}^{-1}A_{1j}=A_{j}\text{ for }j=1,2,3,4.
\]
We consider a second extension of the top tiling in (i), this time yielding a
tiling of $[-1,1]\times\lbrack-1,1],$ by applying $f_{4}^{-1}$ to $\left\{
A_{ij}\right\}  ,$ as illustrated in the left-hand panel in Figure
\ref{ownbgrid}. In this case we notice that%
\begin{equation}
f_{4}^{-1}A_{4j}\subsetneq A_{j}\text{ for }j=1,2,3. \label{attention}%
\end{equation}

This example may be developed further to yield obvious tilings, by squares, of
the plane, half-plane and quarter plane, depending on the choice of
$\mathbf{j}$, typifying aspects of self-similar polygonal tilings, as
discussed in \cite{polygon}, provided by the formula%
\[
\Pi(\mathbf{j})=\lim_{n\rightarrow\infty}f_{-\mathbf{j}|k}\{f_{\mathbf{i}%
|k+1}(A):\mathbf{i\in}\Sigma\}.
\]
A similar development applies to any IFS that obeys the OSC. See for example
\cite{polygon}. Our main purpose in this simple example is to draw attention
to Equation \ref{attention} and the notion of pre-tiles.

\begin{versiona}
\begin{figure}[ptb]%
\centering
\includegraphics[
height=3.1in,
width=5in
]%
{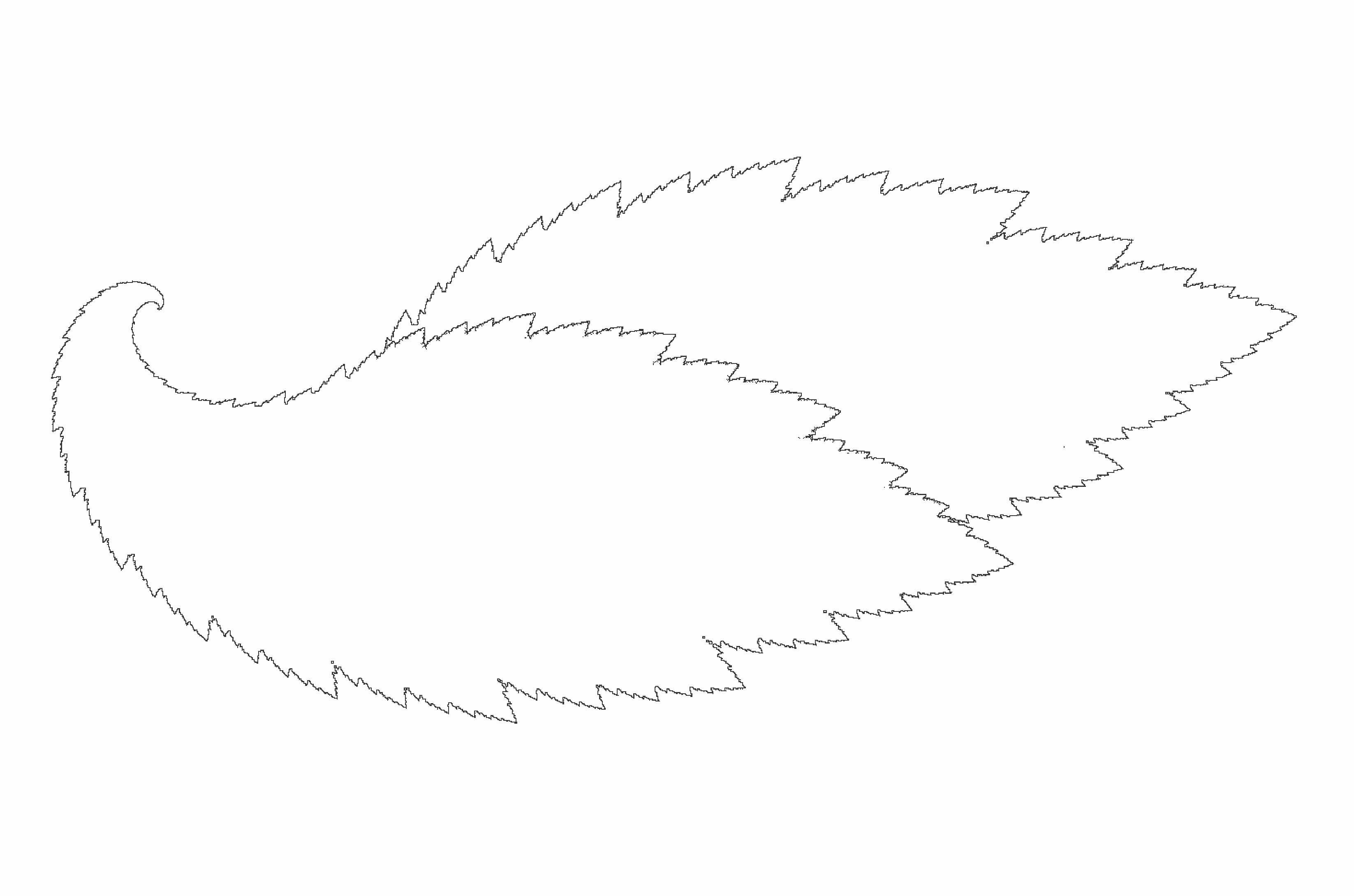}%
\caption{The overlapping attractor of an IFS of two similitudes, each with the same scaling factor. The figure has been rotated ninety degrees anticlockwise relative to the usual presentation of the x,y coordinate system.}%
\label{l0}%
\end{figure}
\end{versiona}

\subsection{\label{notoscsec}Illustrations of pre-tiles when the OSC does not
hold}

We consider the overlapping IFS%
\begin{equation}
\{\mathbb{C};f_{1}(z)=\frac{3}{5}z,f_{2}(z)=\frac{3z+2}{5},f_{3}%
(z)=\frac{3z+2i}{5},f_{4}(z)=\frac{3z+2+2i}{5}\}. \label{ex2formula}%
\end{equation}
The left hand panel in Figure \ref{plaintiles} illustrates the attractor,
namely the unit square, with the four pre-tiles
\begin{align*}
A_{1}  &  =[0,0.6]\times\lbrack0,0.6],A_{2}=(0.6,1]\times\lbrack0,0.6],\\
A_{3}  &  =[0,0.6]\times(0.4,1],A_{4}=(0.6,1]\times(0.6,1].
\end{align*}
The subsequent panels illustrate successively the partitions $f_{1}%
^{-1}(\left\{  A_{i_{1}i_{2}}:i_{1}i_{2}\in\Sigma_{2}\right\}  )$, $f_{1}%
^{-1}f_{4}^{-1}(\left\{  A_{i_{1}i_{2}i_{3}}:i_{1}i_{2}i_{3}\in\Sigma
_{3}\right\}  ,f_{1}^{-1}f_{4}^{-1}f_{1}^{-1}(\left\{  A_{i_{1}i_{2}i_{3}%
i_{4}}:i_{1}i_{2}i_{3}i_{4}\in\Sigma_{4}\right\}  ).$ In each panel the same
viewing window is illustrated in red. It can be seen that in successive
panels, pre-tiles that cover approximately the same region either stay the
same or shrink, while new pre-tiles appear. Figure \ref{plaintilesfinal}
illustrates a close-up of the part of the partition $f_{1}^{-1}f_{4}^{-1}%
f_{1}^{-1}f_{4}^{-1}(\left\{  A_{i_{1}i_{2}i_{3}i_{4}i_{5}}:i_{1}i_{2}%
i_{3}i_{4}i_{5}\in\Sigma_{5}\right\}  )$ that lies within the bounding box.
This patch of the final tiling has now converged, and is a subset of
$f_{-\mathbf{k}|n}(\left\{  A_{\mathbf{i|}n+1}:\mathbf{i}\in\Sigma\right\}  )$
for all $\mathbf{k}$ such that $\mathbf{k}|4=1414.$

\section{EXAMPLES\ OF\ TOPS\ TILINGS}

\begin{versiona}
\begin{figure}[ptb]%
\centering
\includegraphics[
height=1.68in,
width=5.5in
]%
{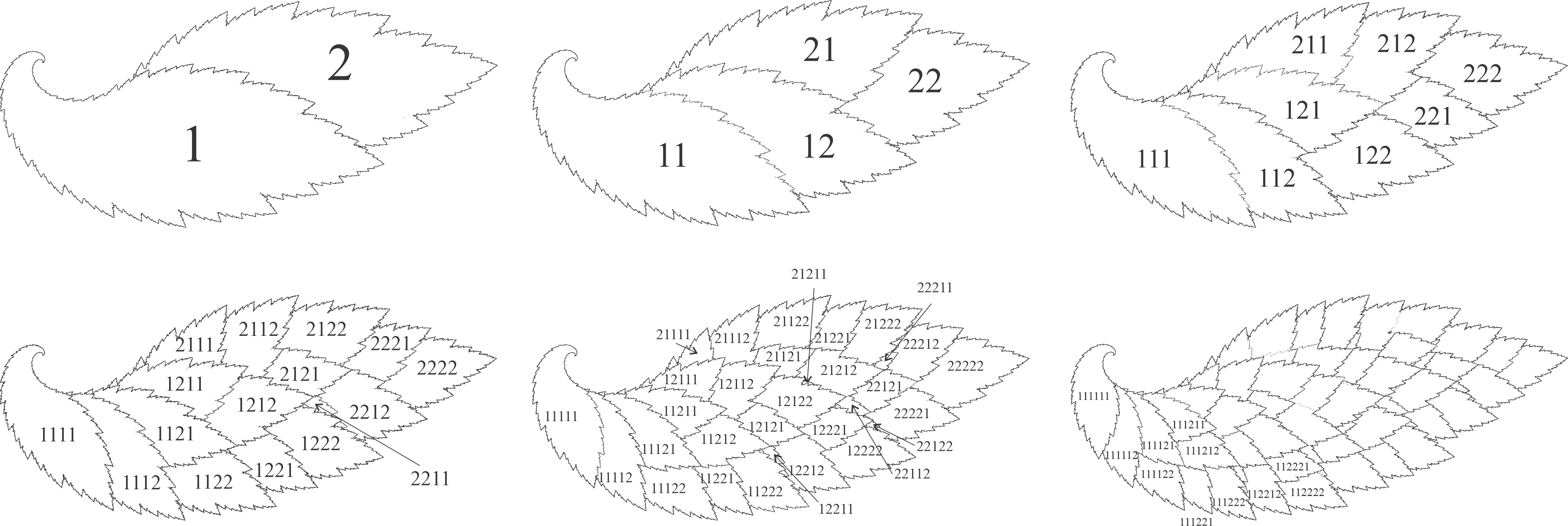}%
\caption{Illustration of the top of the leaf attractor  at depths 0 to 5.}%
\label{overlapleafQ}%
\end{figure}
\end{versiona}

\subsection{Leaf example of a two-dimensional top tiling\label{leafysec}}

\begin{versiona}
\begin{figure}[ptb]%
\centering
\includegraphics[
height=2.7648in,
width=5.0047in
]%
{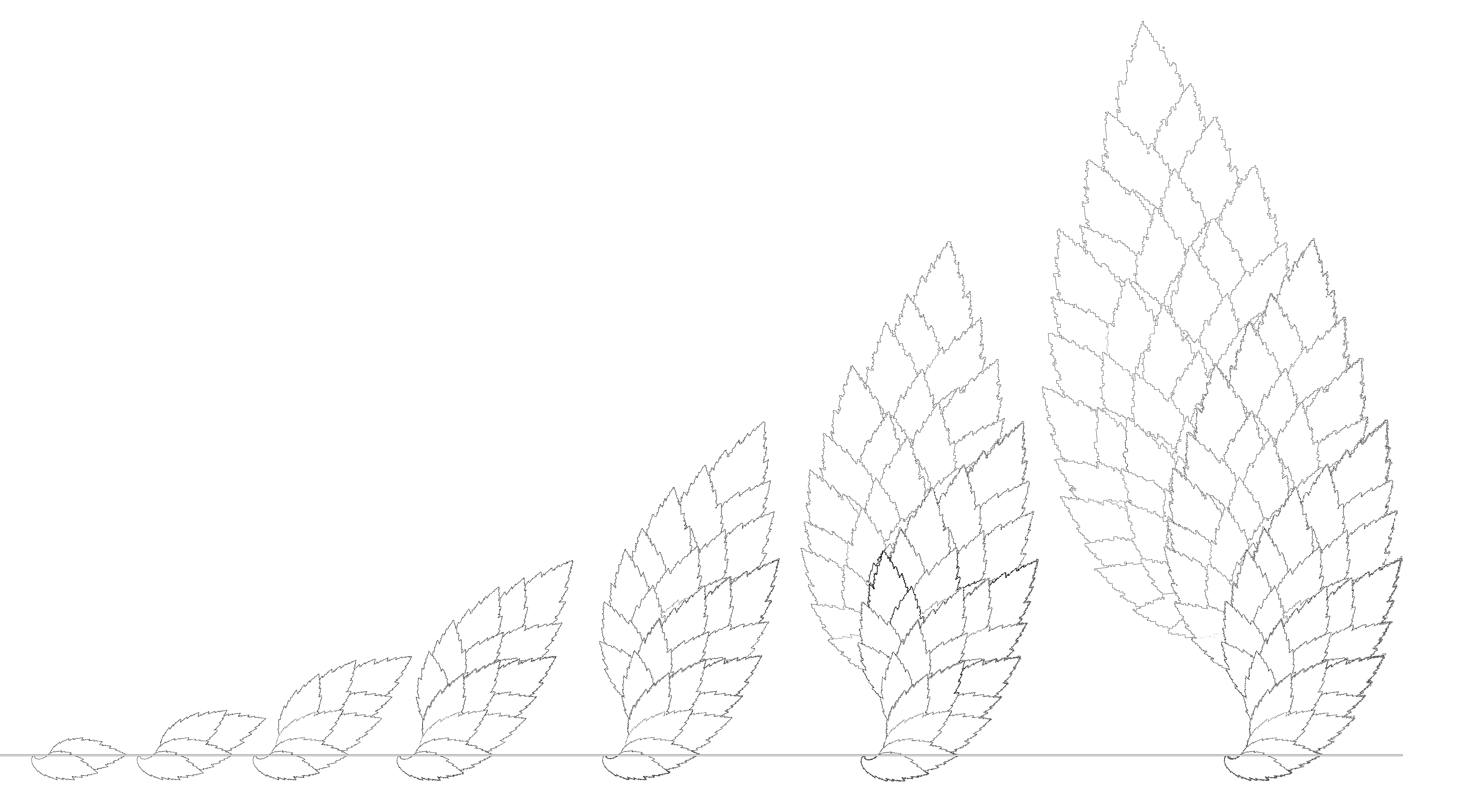}%
\caption{This shows the sequence of tops $\Pi(111...|n)$ for
$n=0,1,\dots,6$ for the leaf IFS.  In each case the tip of the stem is at the same point in the xy-plane.}%
\label{newseq}%
\end{figure}
\end{versiona}

We consider the IFS defined by the two similitudes%
\begin{align}
f_{1}(x,y)  &  =(0.7526x-0.2190y+0.2474,0.2190x+0.7526y-0.0726)
\label{leafifs}\\
f_{2}(x,y)  &  =(-0.7526x+0.2190y+1.0349,0.2190x+0.7526y+0.0678)\nonumber
\end{align}
The attractor looks like a leaf, and is the union of two copies of itself, as
illustrated in Figure \ref{l0}$.$ The proof that this attractor has nonempty
interior is illustrated in Figure \ref{both}. The point with address
$\overline{\mathbf{1}}=111\dots$ is represented by the tip of the stem of the
leaf. The stem is actually arranged in an infinite spiral, not visible in the picture.

\begin{versiona}
\begin{figure}[ptb]%
\centering
\includegraphics[
height=1.7443in,
width=3.8043in
]%
{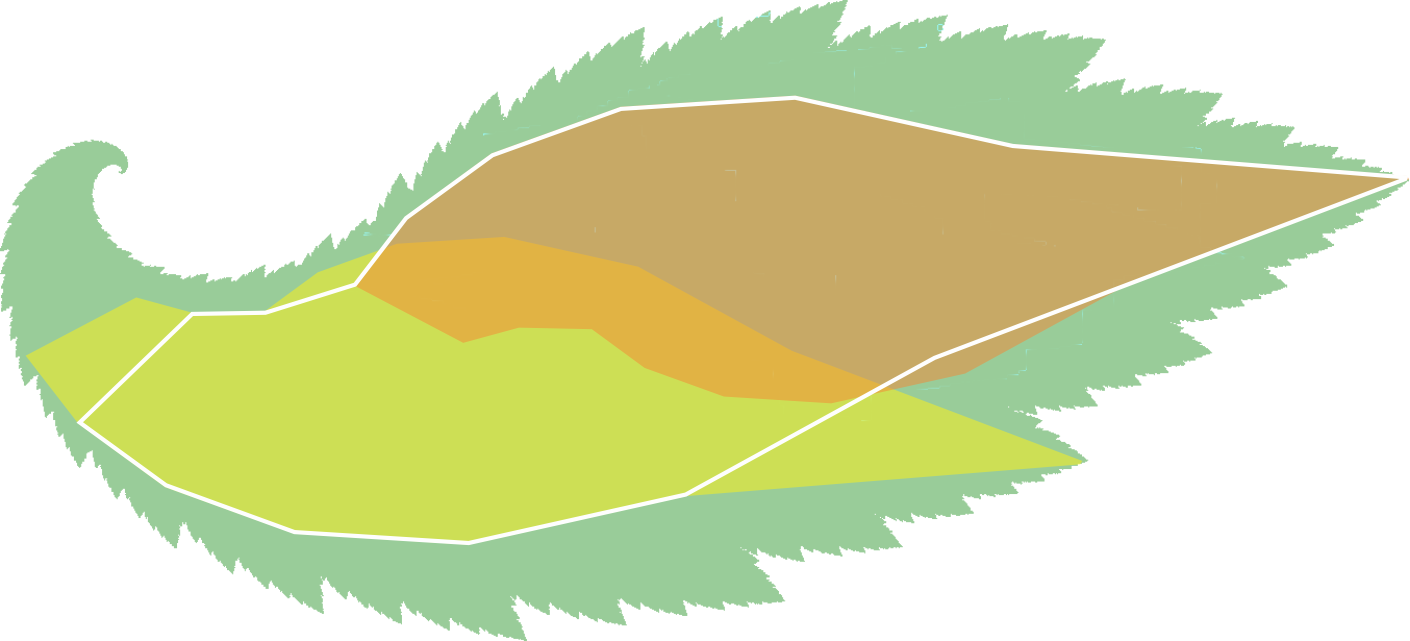}%
\caption{Illustration of the proof that the leaf attractor has nonempty interior. The white polygon is mapped by the IFS into two polygons whose interiors cover the interior of the original polygon.}%
\label{both}%
\end{figure}
\end{versiona}

Figure \ref{overlapleafQ} illustrates the top of $L$ at depths $n\in
\{0,1,\dots,5\}$ labelled for the most part by the addresses in $\Sigma_{n}$ .

Figure \ref{newseq} illustrates the successive top tilings $\Pi(\overline
{1}|n)$ for $n=0,1,\dots,5$ for the IFS in Equation \ref{leafifs}. In this
case the pre-tiles do not shrink for some initial set of levels because
$\overline{1}$ is the lexicographically maximal address. The pictures
represent successively larger fully converged parts of the unbounded tiling
$\Pi(\overline{1}).$ They were computed by the same method as those in Figure
\ref{8x8}.

\begin{versiona}
\begin{figure}[ptb]%
\centering
\includegraphics[
height=5.0in,
width=5.0in
]%
{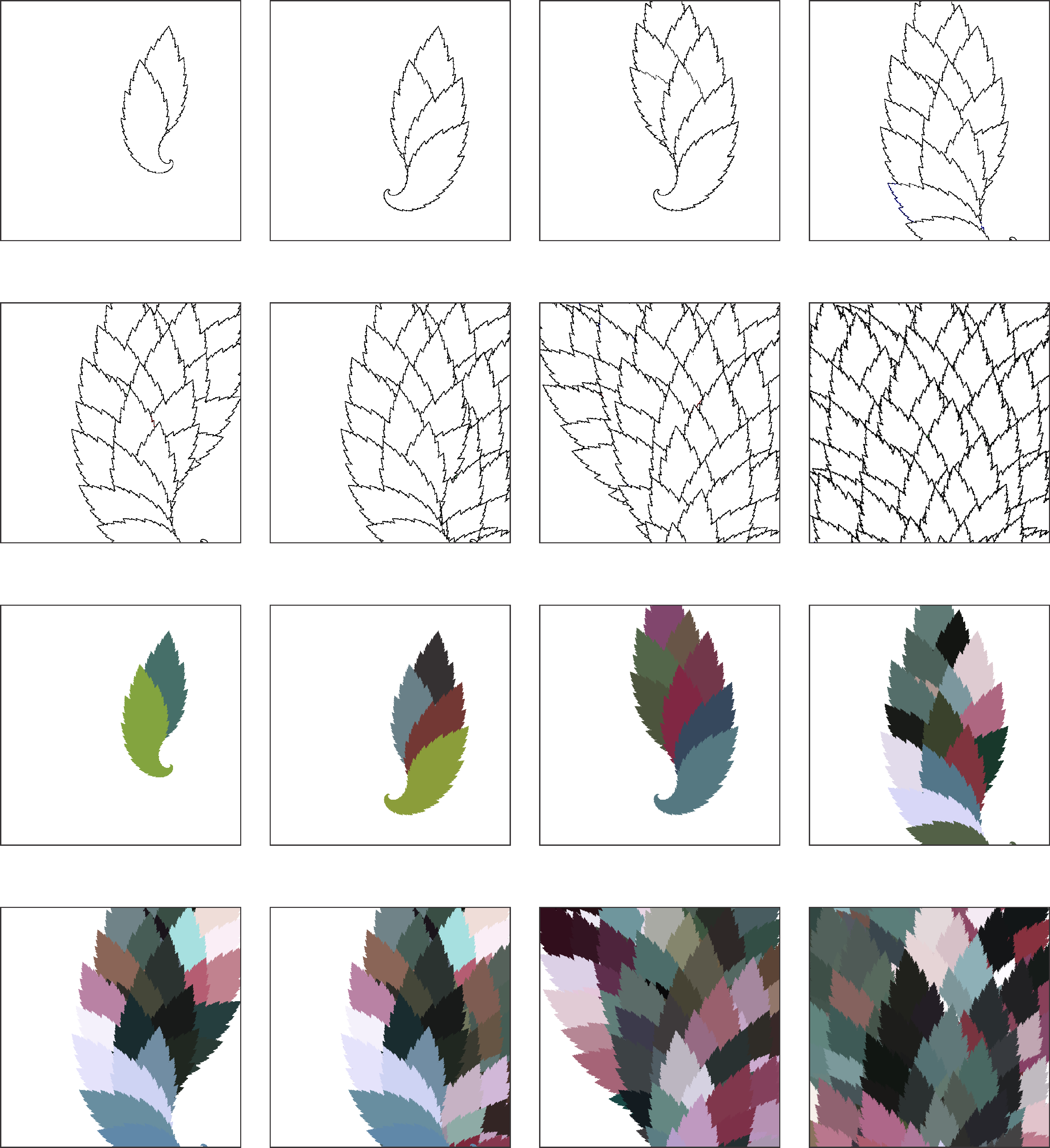}%
\caption{Example of a sequence of pre-tilings associated with a non-reversible address: it illustrates the part of $\Pi(\overline{21}|n)$ that
intersects the square $[0,1]\times\lbrack0,1],$ for $n=0,1,2,3,4,5,6,7.$ See text. }%
\label{8x8}%
\end{figure}
\end{versiona}

Figure \ref{8x8} illustrates the part of $\Pi(\overline{21}|n)$ that
intersects the square $[0,1]\times\lbrack0,1],$ for $n=0,1,2,3,4,5,6,7.$ These
illustrations were computed using the chaos game algorithm, assigning unique
colours "stolen" from a master image, to identify pre-tiles by their addresses
to depth $n,$ illustrated in the bottom eight panels, then using an edge
detection routine. While it appeared from these pictures that $\overline
{21}\in\overleftarrow{\Sigma},$ we found later that $\Pi(\overline{21}|8)$
contains a new tiles in $[0,1]\times\lbrack0,1]$ that overlap many tiles in
$\Pi(\overline{21}|6)$. A subsequent calculation convinced us that
$\overline{21}\notin\overleftarrow{\Sigma}.$

However Figure \ref{bigleaf4} correctly represents the part of the top tiling
$\Pi(\overline{2})$ that lies within $[0,1]\times\lbrack0,1].$ In this case it
is readily checked that $\overline{2}\in\overleftarrow{\Sigma}.$

\begin{versiona}
\begin{figure}[ptb]%
\centering
\includegraphics[
height=5.0in,
width=5.0in
]%
{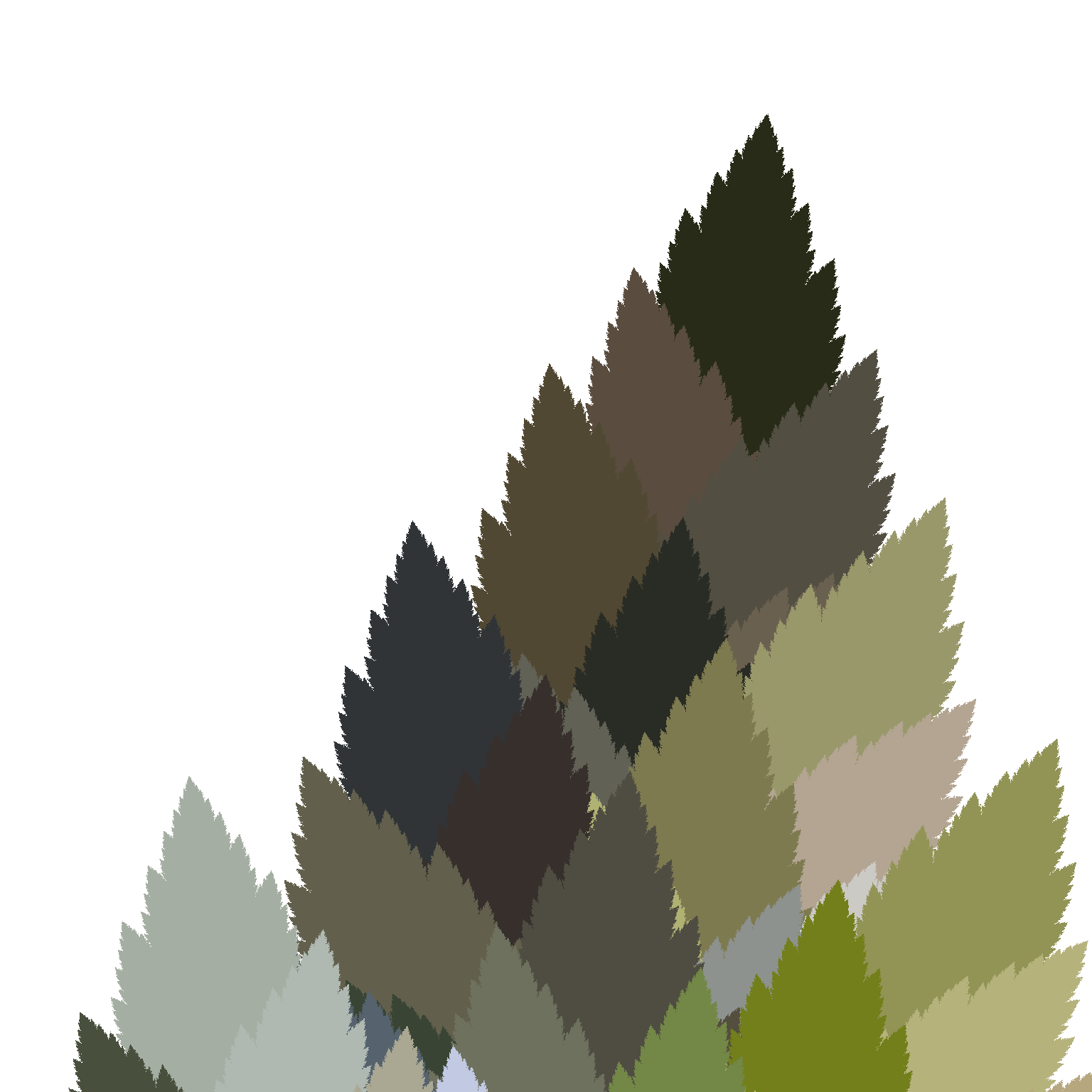}%
\caption{See text. This illustrates the part of the tiling $\Pi(\overline{2})$ within the window  $[0,1]\times\lbrack0,1].$}%
\label{bigleaf4}%
\end{figure}
\end{versiona}

\begin{versiona}
\begin{figure}[ptb]%
\centering
\includegraphics[
height=2.0in,
width=4.5in
]%
{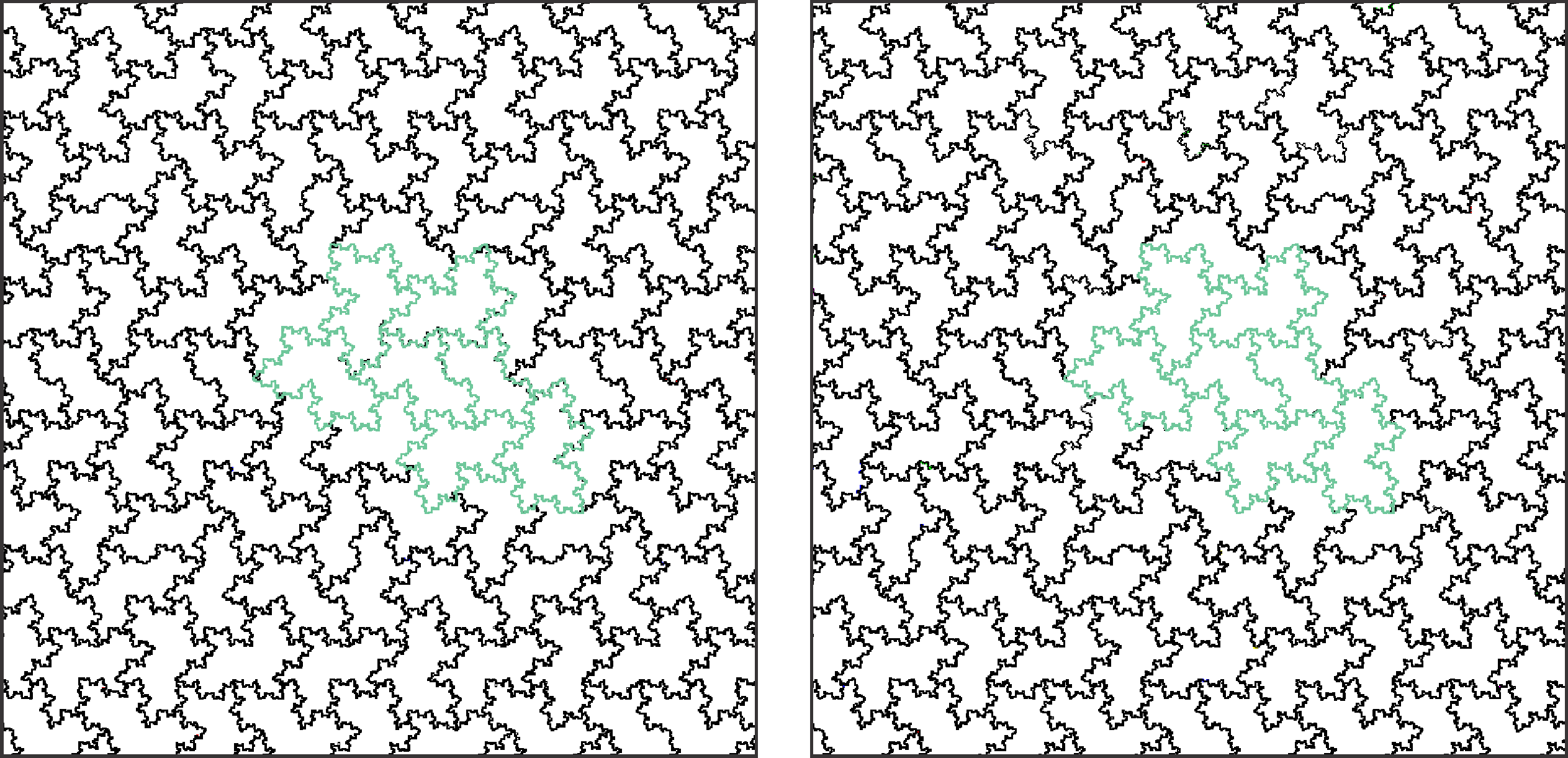}%
\caption{ Within the same window, two of many different top tilings associated with an aperiodic monotile of Smith et al. The top tiling of the attractor of the IFS is illustrated in green. See text. }%
\label{louisaboth}%
\end{figure}
\end{versiona}

\subsection{Top tilings related to hat tilings}

Figure \ref{louisaboth} illustrates the attractor and the part of each of the
top tilings $\Pi(\overline{5324})$ and $\Pi(\overline{2345})$ for the
following IFS, within the viewing window $-1\leq x\leq2,$ $-1.5\leq y\leq1.5.$
The tilings differ in the top left quadrant.

$f_{1}(z)=\frac{2}{\sqrt{5}+3}z$

$f_{2}(z)=\frac{2}{\sqrt{5}+3}e^{(-i\pi/3)}z+\frac{1}{4}\left(  \sqrt
{5}-1\right)  \left(  i\sqrt{3}+1\right)  $

$f_{3}(z)=\frac{2}{\sqrt{5}+3}e^{(-i2\pi/3)}z+\frac{1}{4}\left(  \sqrt
{5}-1\right)  \left(  i\sqrt{3}+3\right)  $

$f_{4}(z)=\frac{2}{\sqrt{5}+3}e^{(2i\pi/3)}z+\frac{1}{4}\left(  i\sqrt
{3}+\frac{1}{4}\sqrt{5}\right)  +\frac{3}{4}-\frac{1}{4}i\sqrt{15}$

$f_{5}(z)=\frac{2}{\sqrt{5}+3}e^{(i\pi/3)}z+\frac{1}{4}(i\sqrt{3}-\sqrt
{5})+\frac{5}{4}-\frac{1}{4}i\sqrt{15}$

$f_{6}(z)=\frac{2}{\sqrt{5}+3}z+\frac{1}{2}(3-\sqrt{5})$

$f_{7}(z)=\frac{2}{\sqrt{5}+3}z+\frac{1}{2}(\sqrt{5}-1)$

This IFS was derived from study of Figure 2.14 in \cite{smith}. It provides an
alternative description of structure of the limiting fractal tiling shapes
associated with hat tilings. Instead of there being two different structures
involved there is a single shape. The tile corresponding to $f_{7}(z),$ namely
$f_{7}(A),$ has a missing piece because it lies partly underneath $f_{6}(A).$
Otherwise the sets $f_{i}(A)$ are just-touching. In fact this system is of
finite type and may be represented by a graph directed IFS. Here different,
non-homeomorphic tilings are codified by their corresponding reversible
addresses, namely the infinite labelled paths of the transposed corresponding
graph. See Figure \ref{louisaboth}. The tiled attractor of the IFS is
indicated in green. More details concerning this example are included in
\cite{corey}.

\section{Generalization}

The theory developed in this paper applies in the setting of point-fibred IFS,
as defined in \cite{tilings} based on the original work of \cite{kieninger}.
In particular, top tilings may be developed in the context of loxodromic
M\~{o}bius transformations \cite{vince}.

A version of the theory also applies when the tiles are defined using regions
yielded by applying $f_{-\mathbf{j}|n}$ to all images of the boundary of $A$
in $f_{\mathbf{i}|n}(A)$ at depth $n$ for all $\mathbf{i}.$ In this setting
Baire's Theorem provides well-defined tilings of say $\mathbb{R}^{2},$ one for
each of a dense set of choices for $\mathbf{j\in}\left\{  1,2,\dots,M\right\}
^{\mathbb{N}}.$ This allows tilings to be well-defined in the case of
non-algebraic data, possessing some level of self-similarity, and maybe
suitable for modelling images derived from nature.

\section{Acknowledgement}

We thank Louisa Barnsley for help with the artwork in general and for her
substantial contribution to the proof that the leaf attractor has nonempty interior.

\end{document}